\documentclass{amsart}

\pdfoutput=1
\usepackage{amssymb}
\usepackage{amsthm}
\usepackage{amsmath}
\usepackage{graphicx}
\usepackage{subcaption}
\usepackage{float}
\usepackage{enumitem}
\usepackage{color}

\title[A multiscale method for linear elasticity]{A multiscale method for linear elasticity reducing Poisson locking }
\author[Patrick Henning]{Patrick Henning\textsuperscript{1}} 
\author[Anna Persson]{Anna Persson\textsuperscript{2}}

\newtheorem{definition}{Definition}[section]
\newtheorem{theorem}[definition]{Theorem}%[section]
\newtheorem{lemma}[definition]{Lemma}%[section]
\theoremstyle{remark}
\newtheorem{remark}[definition]{Remark}%[section]
\newtheorem*{assump}{Assumptions}
\numberwithin{equation}{section}

\DeclareMathOperator{\supp}{supp}
\DeclareMathOperator{\diam}{diam}
\DeclareMathOperator{\card}{card}
\DeclareMathOperator{\meas}{meas}

\newcommand{\ms}{\mathrm{ms}}
\newcommand{\f}{\mathrm{f}}

\newcommand{\R}{\mathbb{R}}
\newcommand{\B}{\mathcal{B}}

\usepackage{hyperref}

\begin{document}
\begin{abstract}
We propose a generalized finite element method for linear elasticity equations with highly varying and oscillating coefficients. The method is formulated in the framework of localized orthogonal decomposition techniques introduced by M{\aa}lqvist and Peterseim (Math. Comp., 83(290): 2583--2603, 2014). Assuming only $L_\infty$-coefficients we prove linear convergence in the $H^1$-norm, also for materials with large Lam\'{e} parameter $\lambda$. The theoretical a priori error estimate is confirmed by numerical examples.
\end{abstract}

\maketitle
\footnotetext[1]{Department of Mathematics, KTH Royal Institute of Technology, SE-100 44 Stockholm, Sweden.}
\footnotetext[2]{Department of Mathematical Sciences, Chalmers University of Technology and University of Gothenburg, SE-412 96 G\"{o}teborg, Sweden.}
%\footnotetext[3]{Supported by ....}

\section{Introduction}
In this paper we study numerical solutions to linear elasticity equations with highly varying coefficients. Such equations typically occur when modeling the deformation of a heterogeneous material, for instance a composite material. Problems with this type of coefficients are commonly referred to as \textit{multiscale} problems.

The convergence of classical finite element methods based on continuous piecewise polynomials depends on (at least) the spatial $H^2$-norm of the solution $u$. However, for problems with multiscale features this norm may be very large. Indeed, if the coefficient varies at a scale of size $\epsilon$, then $\|u\|_{H^2}\sim \epsilon^{-1}$. Thus, to achieve convergence the mesh size must be small ($h<\epsilon$). In many applications this condition leads to issues with computational cost and available memory. To overcome this difficulty several methods have been proposed, where we refer to \cite{Adbulle06,MR3043485,MR2229849,MR3407262} for multiscale methods particularly addressing elasticity problems.

Generalized finite element methods (GFEM, cf. \cite{Babuska83}) belong to the class of Galerkin methods. Instead of constructing the finite dimensional solution space from standard shape functions, a generalized finite element approach is based on constructing a set of locally supported basis functions (not necessarily piecewise polynomials) that incorporate additional information about the structure of the original problem. This strategy can enhance the local approximation properties significantly. In this paper we propose a GFEM based on the ideas in \cite{Malqvist14}, often referred to as \textit{localized orthogonal decomposition} (LOD). 
The methodology of the LOD arose from the framework of the Variational Multiscale Method (VMM) originally proposed by Hughes et al. \cite{Hughes98,HuS07} as a tool for stabilizing finite element methods that perform bad due to an under-resolution of relevant microscopic data. The stabilization was achieved by using a Petrov-Galerkin formulation of the problem with a standard finite element space as trial space and a generalized finite element space for the test-functions.
The concept was reinterpreted and specialized in \cite{Larson07,Mal11} to elliptic homogenization problems. A short time later, the first rigorous analysis was provided in \cite{Malqvist14} by introducing a $H^1$-stable localized orthogonal decomposition for constructing the test function space. In subsequent works, refined construction strategies were proposed \cite{HeP13,Henning14}.

The LOD framework relies on a decomposition of a high-dimensional solution space into a coarse space (spanned by a set of standard nodal basis functions) and a fine scale detail space that is expressed through the kernel of a projection operator. The generalized finite element basis functions are constructed by adding a \textit{correction} from the detail space to each coarse nodal basis function. The corrections are problem dependent and constructed by solving a partial differential equation in the fine scale part of the space. In \cite{Malqvist14} elliptic equations are considered and it is proven that the corrections decay exponentially for these problems. This motivates a truncation to patches of coarse elements, which allow for efficient computations. The resulting method is proved to be convergent of optimal order. This convergence result does not depend on any assumptions regarding periodicity or scale separation of the coefficients. Since its development, the method has been applied to several other types of equations, see, for instance, semilinear elliptic equations \cite{HMP14}, boundary value problems \cite{Henning14}, eigenvalue problems \cite{Malqvist13,HMP14b}, linear and semilinear parabolic equations \cite{MP15}, the Helmholtz problem \cite{Peterseim16,Gallistl16} and the linear wave equation \cite{Abdulle14}. A review is given in \cite{Peterseim15}.

In this work we consider linear elasticity equations with mixed inhomogeneous Dirichlet and Neumann boundary conditions. We construct corresponding correctors for standard nodal basis functions and prove that they decay exponentially. Moreover, we prove that the resulting generalized finite element method converges with optimal order in the spatial $H^1$-norm. The results are confirmed by a numerical example.

Furthermore, the generalized finite element method proposed in this paper reduces the locking effect that is observed for classical finite elements based on continuous piecewise affine polynomials for nearly incompressible materials. The error bound derived for the ideal method (without localization) is uniform in the Lam\'{e} parameter $\lambda$, i.e., completely locking-free. The error estimate for the final localized method depends on $\lambda$, however not in the usual manner, but only weakly through a term that converges with an exponential rate to zero. In practice, this eliminates the locking-effect.

The paper is organized as follows. In Section~\ref{sec:problem} we formulate the problem, in Section~\ref{sec:num} we define the generalized finite element method and in Section~\ref{sec:loc} we perform the localization of the basis functions. Finally, in Section~\ref{sec:experiments} we provide some numerical examples.

\section{Problem formulation}\label{sec:problem}
Let $d=2,3,$ denote the spatial dimension and let $\mathbb{S}:=\R^{d \times d}_{\text{sym}}$ denote the space of $d\times d$ symmetric matrices over $\R$. On $\mathbb{S}$, we use the double-dot product notation 
\begin{align*}
A:B=\sum_{i,j=1}^d A_{ij}B_{ij}, \quad A,B \in \mathbb{S}.
\end{align*}
The computational domain $\Omega\subseteq\R^d$ is assumed to be a bounded polygonal (or polyhedral) Lipschitz domain describing the reference configuration of an elastic medium. We use $(\cdot,\cdot)_{L_2(\Omega)}$ to denote the inner product on $L_2(\Omega,\R^d)$
\begin{align*}
(v,w)_{L_2(\Omega)}:=\int_{\Omega} v(x) \cdot w(x) \,\mathrm dx, \quad v,w \in L_2(\Omega,\R^d),
\end{align*}
and $\|\cdot\|_{L_2(\Omega)}$ for the corresponding norm. Furthermore, we let $H^1(\Omega,\R^d)$ denote the classical Sobolev space with norm $\|v\|^2_{H^1(\Omega)}:=\|v\|^2_{L_2(\Omega)}+\|\nabla v\|^2_{L_2(\Omega)}$, where $\nabla v \in L_2(\Omega,\R^{d\times d})$, and
\begin{align*} 
\| \nabla v \|^2_{L_2(\Omega)} := \sum_{i,j=1}^d \int_{\Omega} (\partial_i v_j(x))^2 \, \mathrm{d}x,
\quad v \in H^1(\Omega,\R^d).
\end{align*}

Let $u:\Omega \rightarrow \R^d$ denote the displacement field of the elastic medium. Under the assumption of small displacement gradients, the (linearized) strain tensor $\varepsilon(u)$ is given by 
\begin{align*}
\varepsilon_{kl}(u):=\frac{1}{2}(\partial_k u_l + \partial_l u_k), \quad  1\leq k,l \leq d.
\end{align*}
Furthermore, Hooke's (generalized) law states that the stress tensor $\sigma$ is given by the relation
\begin{align*}
\sigma_{ij} = \sum_{k,l=1}^d A_{ijkl}(x) \varepsilon_{kl}(u), \quad 1\leq i,j \leq d,
\end{align*}
where $A$ is a fourth order tensor describing the elastic medium. In this paper we assume that the material is strongly heterogeneous and thus $A$ has multiscale properties. The tensor $A$ is assumed to be symmetric in the sense that $A_{ijkl}=A_{jikl}=A_{ijlk}=A_{klij}$ almost everywhere. 

Cauchy's equilibrium equation now states that 
\begin{align*}
-\nabla \cdot \sigma = f,
\end{align*}
where $f: \Omega \rightarrow \R^d$ denotes the body forces. To formulate the problem of interest we let $\Gamma_D$ and $\Gamma_N$ denote two disjoint Hausdorff measurable segments of the boundary, such that $\Gamma_D \cup \Gamma_N = \partial \Omega$, where Dirichlet and Neumann conditions are imposed respectively. The linear elasticity problem consists of finding the displacement $u$ and the stress tensor $\sigma$ such that
\begin{alignat}{2}
-\nabla \cdot \sigma &= f,& \quad & \text{in } \Omega, \label{elas1}  \\
\sigma_{ij}&=\sum_{k,l=1}^d A_{ijkl} \hspace{2pt} \varepsilon_{kl}(u),& & \text{in } \Omega, \label{elas2}\\
u &= g,& & \text{on } \Gamma_D,\label{bd1}\\
\sigma \cdot  n &= b,& & \text{on } \Gamma_N,\label{bd2}
\end{alignat}
where we assume that $\meas(\Gamma_D)>0$. Here $g,b:\Omega\rightarrow \R^d$ denotes the Dirichlet and Neumann data respectively.

To pose a variational form of problem \eqref{elas1}-\eqref{bd2} we need to define appropriate test and trial spaces. Letting $\gamma:H^1(\Omega)\rightarrow L_2(\Gamma_D)$ denote the trace operator onto $\Gamma_D$, we define the test space
\begin{align*}
V := \{v \in (H^1(\Omega))^d: \gamma v = 0\}.
\end{align*}
Multiplying the equation \eqref{elas1} with a test function from $V$ and using Green's formula together with the boundary conditions \eqref{bd2} we get that
\begin{align*}
(\sigma:\nabla v)_{L_2(\Omega)}=(f,v)_{L_2(\Omega)} + (b,v)_{L_2(\Gamma_N)}.
\end{align*}
Due to the symmetry of $A$ we have the identity $(\sigma:\nabla v)=(\sigma:\varepsilon(v))$, and by defining the bilinear form
\begin{align*}
\B(u,v):=(\sigma:\varepsilon(v))_{L_2(\Omega)}=(A(x)\varepsilon(u):\varepsilon(v))_{L_2(\Omega)},
\end{align*}
we arrive at the following weak formulation of \eqref{elas1}-\eqref{bd2}. Find $u\in H^1(\Omega,\R^d)$, such that $\gamma u = g$, and
\begin{alignat}{2}
\B(u,v)= (f,v)_{L_2(\Omega)} + (b,v)_{L_2(\Gamma_N)},& \quad &\forall v \in V.\label{weak}
\end{alignat}

\begin{remark}\label{isotropic}
In the case of an isotropic medium the elasticity coefficient satisfies $A_{ijkl}=\mu(\delta_{ik}\delta_{jl}+\delta_{il}\delta_{jk}) + \lambda \delta_{ij}\delta_{kl}$, where $\delta_{ij}$ is the Kronecker delta, and $\mu$ and $\lambda$ are the so called Lam\'{e} coefficients. The stress tensor can in this case be simplified to
\begin{align*}
\sigma = 2\mu \varepsilon(u) + \lambda (\nabla \cdot u) I,
\end{align*}
where $I$ is the identity matrix.  
\end{remark} 

\begin{assump} We make the following assumptions on the data
	\begin{enumerate}[label=(A\arabic*)]
		\item \label{ass-coeff} $A_{ijkl} \in L_\infty(\Omega,\R)$, $1\leq i,j,k,l \leq d$, and there exist positive constants $\alpha,\beta \in \R$ such that
		\begin{align*} 
		\alpha B:B \leq A(\cdot)B:B \leq \beta B:B, \quad \forall B\in \mathbb{S}, \quad \text{a.e. in } \Omega. 
		\end{align*}
		\item \label{ass-data} $f\in L_2(\Omega,\R^d)$, $b\in L_2(\Gamma_N,\R^d)$, and $g\in H^{1/2}(\Gamma_D,\R^d)$.
	\end{enumerate}
\end{assump}

Recall Korn's inequality for a domain with mixed boundary conditions, see, for instance, \cite{BrennerScott, Mazzucato10}.
\begin{lemma}[Korn's inequality]\label{korn}
	Let $\Omega \subset \R^d$ denote a bounded and connected Lipschitz-domain, and let $\Gamma_D$ denote the part of the boundary where Dirichlet boundary conditions are defined. If $\meas(\Gamma_D) >0$, then
	\begin{align}
	\label{korn-global-estimate}\| \nabla v \|_{L_2(\Omega)} &\leq C_\mathrm{ko}\| \varepsilon(v) \|_{L_2(\Omega)}, \quad \forall v \in V,	%\label{korn-local-estimate}\| \nabla v \|_{L_2(\omega)} &\leq \sqrt{2} \| \varepsilon(v) \|_{L_2(\omega)}, \quad \forall v \in H^1_0(\omega,\R^d).
	\end{align}
	Here $C_\mathrm{ko}$ is a constant depending only on $\Omega$.
\end{lemma}
In the case $\Gamma_D=\partial \Omega$ we have $C_\mathrm{ko}=\sqrt{2}$, independently of the size of $\Omega$. Using \eqref{korn-global-estimate} we derive the following bounds,
\begin{align}\label{B-bounds}
\alpha C_\mathrm{ko}^{-2}\|\nabla v\|^2_{L_2(\Omega)}\leq \B(v,v) \leq \beta \|\nabla v\|^2_{L_2(\Omega)}, \quad \forall v\in V,
\end{align}
where we have used the bound $\|\varepsilon(v)\|_{L_2(\Omega)}\leq \|\nabla v\|_{L_2(\Omega)}$. It follows that the bilinear form $\B(\cdot,\cdot)$ is an inner product on $V$ and existence and uniqueness of a solution to the problem \eqref{weak} follows from the Lax-Milgram lemma. We denote the norm induced by the inner product $\B(\cdot,\cdot)$ by $\|v \|^2_{\B(\Omega)}:=\B(v,v)$ for $v\in V$.
\begin{remark}
	In the case of an isotropic material (see Remark~\ref{isotropic}) we have the bounds
	\begin{align*}
	C_\mathrm{ko}^{-2} 2\mu_1 \|\nabla v\|^2_{L_2(\Omega)}&\leq \|\sqrt{2\mu} \varepsilon(v)\|^2_{L_2(\Omega)} \leq \|\sqrt{2\mu} \varepsilon(v)\|^2_{L_2(\Omega)} + \|\sqrt{\lambda}\nabla \cdot v\|^2_{L_2(\Omega)} \\&=\B(v,v) \leq C(2\mu_2 + \lambda_2)\|\nabla v\|_{L_2(\Omega)}^2,
	\end{align*}
	where $\mu_1>0$ is the lower bound of $\mu$ and $\mu_2,\lambda_2 \leq \infty$ are the upper bounds of $\mu$ and $\lambda$ respectively. We emphasize that this means that only $\beta$ in \eqref{B-bounds} depends on $\lambda$.
\end{remark}

\section{Numerical Approximation}\label{sec:num}

\subsection{Classical finite element}
First, we define the classical finite element space of continuous and piecewise affine elements. Let $\mathcal{T}_h$ be a regular triangulation of $\Omega$ into closed triangles/tetrahedra with mesh size $h_T:= \diam(T)$, for $T\in\mathcal{T}_h$, and denote the largest diameter in the triangulation by $h:=\max_{T\in \mathcal{T}_h} h_T$. We assume that the family of triangulations \{$\mathcal{T}_h\}_{h>0}$ is shape regular. Now define the spaces
\begin{align*}
S_h&=\{v\in (C(\bar{\Omega}))^d: v|_T \text{ is component-wise a polynomial of degree} \leq 1, \forall T \in \mathcal{T}_h\},\\
V_h &= S_h\cap V.
\end{align*}
Furthermore, we let $\mathcal N_h$ denote the nodes generated by $\mathcal T_h$ and $\mathring{\mathcal N}_h=\mathcal N_h\setminus \Gamma_D$ the free nodes in $V_h$. Now, let $g_h \in S_h$ be an approximation of an extension of $g$, such that $g_h(z)=0$, $\forall z\in \mathring{\mathcal N}_h$ and $\gamma g_h$ is some appropriate approximation of $g$. %and $g_h(z)=g(z)$ $\forall z\in \mathcal N_h \cap \Gamma_D$. 
The classical finite element method now reads; find $u_h=u_{h,0}+g_h$, such that $u_{h,0}\in V_h$ and
\begin{alignat}{2}
\B(u_{h,0},v)= (f,v)_{L_2(\Omega)} + (b,v)_{L_2(\Gamma_N)} - \B(g_h,v),& \quad &\forall v \in V_h.\label{fem}
\end{alignat} 
Note that $\gamma u_h = \gamma g_h$, where $\gamma g_h$ is an approximation of $g$.

\begin{theorem}\label{femconv}
	Let $u$ be the solution to \eqref{weak} and $u_h$ the solution to \eqref{fem}. If the solution $u$ is sufficiently regular we have 
	\begin{align*}
	\|u-u_h\|_{H^1(\Omega)}\leq C_{A}h\|D^2u\|_{L_2(\Omega)},
	\end{align*}
	%where $C_A$ depends on the variations in $A$ 
	where $C_A$ depends on the size of $A$ and $\|D^2u\|_{L_2(\Omega)}$ depends on the variations in $A$ via a regularity estimate $\|D^2u\|_{L_2(\Omega)} \le C(u,\Omega) \| A \|_{W^{1,\infty}(\Omega)}$. In particular, we have $\|D^2u\|_{L_2(\Omega)}\rightarrow \infty$ the faster $A$ oscillates. 
\end{theorem}

%We emphasize that 
Since
the a priori bound in Theorem~\ref{femconv} depends, through the $H^2$-norm of $u$, on the variations (derivatives) in the data,
% For the problems of interest in this paper, this regularity assumption might not be satisfied. Also, if we indeed have $u\in H^2$, the norm $\|D^2u\|$ may be of size $\epsilon^{-1}$ if the coefficients oscillates with a frequency of $\epsilon^{-1}$.
%Thus, 
the mesh width $h$ must be sufficiently small for $u_h$ to be a good approximation of $u$. In the context of multiscale problems, this results in a significant computational complexity.
%It is essential for the method proposed in the next subsection
In the following we assume
that $h$ is small enough and we shall refer to $u_h$ as a reference solution. However, we emphasize that our method never requires to compute this expensive reference solution and that it is purely used for comparisons.
%Note that the solution $u_h$ is expensive to compute for this size of $h$.

\subsubsection{Poisson locking}\label{locking}
	This subsection describes the phenomenon known as locking, sometimes referred to as Poisson locking to distinguish it from other types of locking. To simplify the discussion here we assume that we have an isotropic material with $\mu$ and $\lambda$ \textit{constant} parameters and $g_D=0$ on $\Gamma_D=\partial \Omega$. In this case we can exploit Galerkin orthogonality and the norm-equivalence in Remark \ref{isotropic} to see that the error bound in Theorem~\ref{femconv} becomes the estimate
	\begin{align}\label{error-locking}
	\|u-u_h\|_{H^1(\Omega)}\leq Ch\frac{\sqrt{2\mu+\lambda}}{\sqrt{2 \mu}}\|D^2u\|_{L_2(\Omega)},
	\end{align}
	where $C$ is independent of $\mu$ and $\lambda$. Moreover, $\|D^2u\|_{L_2(\Omega)}$ is independent of $\mu$ and $\lambda$ which follows from the stability estimate (see \cite{BrennerSung92}),
	\begin{align}
	\label{stability-estimate-brenner}\|u\|_{H^2(\Omega)} + \lambda\|\nabla\cdot u\|_{H^1(\Omega)} \leq C_\Omega\|f\|_{L_2(\Omega)},
	\end{align}
	where $C_\Omega$ is independent of $\mu$ and $\lambda$. We emphasize that the estimate \eqref{stability-estimate-brenner} does not hold if $\mu$ and $\lambda$ vary in space. Since both $C$ and $\|D^2u\|_{L_2(\Omega)}$ in \eqref{error-locking} are independent of $\lambda$, we conclude that the error bound blows up as $\lambda\rightarrow \infty$. This is counter-intuitive to the observation that the error with respect to the $H^1$-best-approximation in $V_h$ is not affected by $\lambda$.

%However, from the stability estimate we know that the exact solution can be bounded independently of the size of $\lambda$. This implies that the best-approximation in $V_h$ is independent of $\lambda$ and we have
%\begin{align}
%\label{best-approx-est}\inf_{v_h \in V_h} \|u-v_h\|_{H^1(\Omega)}\le 
%\|u-I_h(u)\|_{H^1(\Omega)}\le
%Ch \|D^2u\|_{L_2(\Omega)} \le C_\Omega h \|f\|_{L_2(\Omega)},
%\end{align}
%where $I_h$ denote a suitable interpolation operator. Comparing estimate \eqref{error-locking} and \eqref{best-approx-est} we face the paradox of {\it poisson locking}. Even though the $P1$ Lagrange finite element space $V_h$ is rich enough so that we could observe a linear convergence in $h$ (for the $H^1$-best approximation) for any mesh resolution with $h\lesssim 1$, we practically require the more severe condition $h\lesssim \lambda^{-1/2}$ to observe the desired convergence. This becomes crucial for $\lambda \rightarrow \infty$.
%This contradicts the observation that we expect $\| u_h \|_{H^1(\Omega)}$ to explode for $\lambda \rightarrow \infty$ (since \eqref{error-locking} is optimal). 
In fact, there is a simple reason for this phenomenon. For $\lambda \rightarrow \infty$ we have that the displacement must fulfill the extra condition $\nabla \cdot u = 0$. However, $v_h=0$ is the only function in $V_h$ that fulfills $\nabla \cdot v_h=0$. This forces the Galerkin-approximation $u_h$ to convergence to the bad approximation $u_h=0$ in order to remain stable. This issue can be avoided by using discrete solution spaces in which divergence-free functions can be well-approximated, cf. the robust methods in \cite{BrennerScott,BrennerSung92, Babuska92,Arnold07}, where it is in fact possible to derive estimates of the type $\|u-u_h\|_{H^1(\Omega)}\leq Ch\|D^2u\|_{L_2(\Omega)}$ independent of $\lambda$.% as we have it for the best-approximation.

%	The finite element method based on continuous piecewise affine elements is sufficient for fix $\mu$ and $\lambda$ of reasonable size. However, as $\lambda \rightarrow \infty$ (incompressible materials) the problem of locking in the finite element solution occurs.
%, since the constant $C$ in Theorem~\ref{femconv} depends on the upper bound of $\lambda$.
	%Indeed, it is proven in \cite{BrennerScott} that for any mesh size $h$, if $\lambda$ is sufficiently large, there is $u\in H^2(\Omega)\cap H^1_0(\Omega)$ such that the relative error $\|u-u_h\|_{H^1(\Omega)}/\|\nabla\cdot\sigma(u)\|_{L_2(\Omega)}$ is bounded from below by a positive constant independent of $h$. 
	
%$\\$
From the discussion above we conclude that
%In practice, this means that
if $\lambda$ is large compared to $\mu$ the mesh size must be sufficiently small, i.e. $h\lesssim 1/\sqrt{\lambda}$, to achieve convergence for conventional Lagrange $P1$ finite elements. A natural question is what the typical ranges of values for $\mu$ and $\lambda$ are and how they are related. The Lam\'e parameters are determined by Young's modulus $E$ and Poisson's ratio $\nu$ according to $\mu=\frac{E}{2(1+\nu)}$ and $\lambda=\frac{E \nu}{(1+\nu)(1-2\nu)}$. Consequently, we obtain $\frac{\sqrt{2\mu+\lambda}}{\sqrt{2 \mu}} = \sqrt{\frac{1}{1-2\nu}}$ and hence \eqref{error-locking} reduces to
\begin{align}\label{error-locking-2}
\|u-u_h\|_{H^1(\Omega)}\leq C_\Omega \frac{h}{\sqrt{1-2\nu}}\|f\|_{L_2(\Omega)},
\end{align} 
where we see that the problem only arises if the Poisson's ratio is close to $\nu=0.5$, which describes a perfectly incompressible material. In most engineering applications the value of Poisson's ratio lies between $0.2$ and $0.35$ (e.g. $\nu = 0.27-0.30$ for steel, $\nu = 0.2-0.3$ for rocks such as granite or sandstone and $\nu = 0.17-0.27$ for glass; cf. \cite{GeG09}). Poisson's ratios larger than $0.45$ are rare. Examples for such tough cases are clay ($\nu \le 0.45$), gold ($\nu=0.45$) and lead ($\nu=0.46$). Natural rubber with $\nu = 0.4999$ can be considered as the most extreme case (cf. \cite{MoR09}). These values give us a clear image about the order of magnitude required for $h$ in practical scenarios. If the extension of $\Omega$ is of order $1$, tough cases ($\nu\approx 0.45$) require $h\lesssim \frac{1}{3}$ and extreme cases ($\nu\approx 0.4999$) require $h\lesssim \frac{1}{70}$. These values help us to understand the phenomenon of locking better. The constraints that are imposed by Poisson locking are not severe (in the sense that it does typically not make the problem prohibitively expensive), but they are highly impractical and not desirable in the sense that they make the problem significantly more expansive than it should be. For instance for $\nu = 0.45$ the mesh needs to be three times finer than for a locking-free method, which makes an enormous difference in CPU demands due to the curse of dimension.
%(or for the best approximation)

\subsubsection{Poisson locking for multiscale problems}\label{locking-multiscale}

This paper is devoted to multiscale problems and the locking effect has to be seen from a different perspective in this case. Multiscale elasticity problems as they typically arise in engineering or in geosciences involve material parameters (in general form represented by the tensor $A(x)$) that vary on an extremely fine scale $\epsilon$ (relative to the extension of the computational domain) with $\epsilon \ll \lambda^{-1/2}$. These variations need to be resolved by an underlying fine mesh which imposes the condition $h < \epsilon \ll \lambda^{-1/2}$ even for locking-free methods. In other words, the natural constraints imposed by the variations of the coefficient are much more severe than the constraints imposed by the locking effect. Since we assume that the reference solution $u_h$ given by \eqref{fem} is a good approximation to our original multiscale problem (i.e. $h<\epsilon$), then the solution will not suffer from the locking effect either. For that reason we consider $u_h$ as being locking-free. Our multiscale method is constructed to approximate $u_h$ on significantly coarser scales of order $H$, and we call this method a {\it locking-free multiscale method} if the convergence rates in $H$ are independent of $\lambda$ and the variations of $A$.

Locking and multiscale are two different characteristics that typically need to be treated with different approaches, as a multiscale method is not necessarily locking-free. In the following we show that the framework of the LOD can be used for stabilizing $P1$ Lagrange finite elements in such a way that both effects are reduced simultaneously. In particular we show that it is not necessary to use higher order Lagrange elements, discontinuous Galerkin approaches, mixed finite elements or Crouzeix-Raviart finite elements as they are commonly required for eliminating Poisson locking.
	
In this paper the error estimate for the ideal method (without localization) in Lemma~\ref{gfemconv} is independent of $\lambda$ and thus locking-free. The localization depends on the contrast $\beta/\alpha$, see Theorem~\ref{locgfemconv}. However, this ratio enters only through a term that converges with exponential order to zero. Consequently, the locking effect decays exponentially in the localized method. This is also tested numerically in Section~\ref{sec:experiments}.

\subsection{Generalized finite element}
In this subsection we introduce a generalized finite element method. Let $V_H$ denote the same classical finite element space as $V_h$, but with a coarser mesh size $H>h$. Let $\mathcal{T}_H$ be the triangulation associated with the space $V_H$ and assume that $\mathcal{T}_h$ is a refinement of $\mathcal{T}_H$ such that $V_H \subseteq V_h$. In addition to shape regular, we assume the family $\{\mathcal T_H\}_{H>h}$ to be quasi-uniform. 

We define $\mathcal{N}_H$ and $\mathring{\mathcal{N}}_H$ analogously to $\mathcal{N}_h$ and $\mathring{\mathcal{N}}_h$. Note that the mesh width $H$ is too coarse for the classical finite element solution \eqref{fem} in $V_H$ to be a good approximation. The aim is now to define a new (multiscale) space with the same dimension as $V_H$, but with better approximation properties. 

To define such a multiscale space we need to introduce some notation. First, let $I_H: V_h\rightarrow V_H$ denote an interpolation operator with the property that $I_H \circ I_H = I_H$ and
\begin{align}\label{interpolation-local}
H^{-1}_T\|v-I_Hv\|_{L_2(T)} + \|\nabla I_Hv\|_{L_2(T)} \leq C_I\|\nabla v\|_{L_2(\omega_T)}, \quad \forall T\in \mathcal{T}_H, \ v \in V_h,
\end{align}
where
\begin{align*}
\omega_T:=\cup\{\hat T \in \mathcal T_H: \hat T \cap T\neq \emptyset\}.
\end{align*}
For a shape regular mesh, the estimates in \eqref{interpolation-local} can be summed to a global estimate
\begin{align}\label{interpolation-global}
H^{-1}\|v-I_Hv\|_{L_2(\Omega)}+ \|\nabla I_Hv\|_{L_2(\Omega)} \leq C_\rho\|\nabla v\|_{L_2(\Omega)},
\end{align}
where $C_\rho$ depends on $C_I$ and the shape regularity parameter,  $\rho>0$;
\begin{align*}
\rho:=\max_{T \in \mathcal{T}_H} \rho_T, \ \text{with} \ \rho_T:= \frac{\diam B_T}{\diam T}, \ \text{for}\ T\in \mathcal{T}_H.
\end{align*}
Here $B_T$ is the largest ball contained in $T$. For instance, we could choose $I^i_H=E^i_H\circ \Pi^i_H$, $1\leq i\leq d$, where $\Pi^i_H$ is the $L_2$-projection onto $P_1(\mathcal{T}_H)$, the space of functions that are affine on each triangle $T\in \mathcal{T}_H$ and $E^i_H:P_1(\mathcal{T}_H)\rightarrow V_H$ the averaging operator defined by 
\begin{align}\label{interpolation}
(E^i_H(v))(z) = \frac{1}{\card\{T\in \mathcal{T}_H: z \in T\}}\sum_{T\in \mathcal{T}_H: z \in T} v|_T(z),
\end{align} 
where $z\in \mathring{\mathcal N}_H$, see \cite{Peterseim15} for further details and other possible choices of $I_H$. %Note that $I_H=(I^1_H,...,I^d_H)$.

Let $V_\f$ denote the kernel to the operator $I_H$
\begin{align*}
V_\f:=\ker I_H=\{v\in V_h: I_Hv=0\}.
\end{align*}
The space $V_h$ can now be split into the two spaces $V_h=V_H\oplus V_\f$, meaning that $v_h\in V_h$ can be decomposed into $v_h=v_H+v_\f$, such that $v_H\in V_H$ and $v_\f \in V_\f$. The kernel $V_\f$ is a detail space in the sense that it captures all features that are not captured by the (coarse) space $V_H$.

Let $R_\f: V_h \rightarrow V_\f$ be the Ritz projection onto $V_\f$ using the inner product $\B(\cdot,\cdot)$ such that
\begin{align}\label{Ritz-Vf}
\B(R_\f v , w) = \B(v,w), \quad \forall w \in V_\f, \quad v\in V_h.
\end{align}
Since $v_h=v_H+v_\f$ with $v_H\in V_H$ and $v_\f \in V_\f$ we have
\begin{align*}
v_h -R_\f v_h = v_H - R_\f v_H, \quad \forall v_h\in V_h,
\end{align*}
and we define the multiscale space
\begin{align}\label{msbasis}
V_\ms = \{v_H-R_\f v_H : v_H \in V_H\}.
\end{align}
Note that this space has the same dimension as $V_H$, but contains fine scale features. Indeed, with $\lambda_z$ denoting the hat basis function in $V_H$ corresponding to node $z$, the set 
\begin{align*}
\{\lambda_z-R_\f \lambda_z : z \in \mathring{\mathcal N}_H\},
\end{align*}  
is a basis for $V_\ms$. Moreover, we note that $V_\ms$ is the orthogonal complement to $V_\f$ with respect to the inner product $\B(\cdot,\cdot)$. Thus the split $V_h=V_\ms\oplus V_\f$ and the following orthogonality holds for $v_\ms \in V_\ms$ and $v_\f \in V_\f$
\begin{align}\label{B-orthog}
\B(v_\ms,v_\f)=\B(v_\f,v_\ms)=0.
\end{align} 

To define a generalized finite element method we aim to replace the space $V_h$ with $V_\ms$ in \eqref{fem}. Due to the inhomogeneous boundary conditions we also need two extra corrections similar to the ones used in \cite{Henning14}. For the Dirichlet condition we subtract $R_\f g_h$ from the solution. For the Neumann condition we define a correction $\tilde b_\f  \in V_\f$ such that
\begin{align}\label{neumann-corr}
\B(\tilde b_\f , w) = (b,w)_{L_2(\Gamma_N)}, \quad \forall w \in V_\f. 
\end{align}

We are now ready to define the generalized finite element method; find 
\begin{align*}
u_\ms=u_{0,\ms} + \tilde b_\f  +  g_h - R_\f g_h,
\end{align*} 
such that $u_{0,\ms} \in V_\ms$ and
\begin{align}\label{gfem}
\B(u_{0,\ms},v)=(f,v)_{L_2(\Omega)} + (b,v)_{L_2(\Gamma_N)}- \B(\tilde b_\f  + g_h - R_\f g_h,v), \quad \forall v \in V_\ms.
\end{align}
Note that both $\tilde b_\f =R_\f g_h=0$ on $\Gamma_D$, so $\gamma u_\ms = \gamma g_h$, and %due to the orthogonality \eqref{B-orthog} 
\begin{align*}
\B(u_\ms,v) = (f,v)_{L_2(\Omega)} + (b,v)_{L_2(\Gamma_N)}, \quad \forall v\in V_\ms,
\end{align*}
as desired.

\begin{lemma}\label{gfemconv}
	Let $u_h$ be the solution to \eqref{fem} and $u_\ms$ the solution to \eqref{gfem}. Then
	\begin{align}\label{gfemconv-est}
	\|u_h-u_\ms\|_{H^1(\Omega)} \leq CH\alpha^{-1}\|f\|_{L_2(\Omega)},
	\end{align}
	where $C$ depends on $C_\mathrm{ko}$ and $C_\rho$.
\end{lemma}
\begin{proof}
	Define $e:=u_h-u_\ms$. Since $V_\ms \subseteq V_h$, we have the Galerkin orthogonality
	\begin{align*}
	\B(e,v) = 0,\quad \forall v \in V_\ms.
	\end{align*}
	Recall that we can write $e=(I-R_\f)e + R_\f e$ where $(I-R_\f)e \in V_\ms$ and $R_\f e\in V_\f$. Using this we get
	\begin{align*}
	\alpha C_\mathrm{ko}^{-2} \|\nabla e\|^2_{L_2(\Omega)} &\leq \B(e,e) = \B(e,R_\f e) = \B(u_h-u_\ms,R_\f e) \\&= (f,R_\f e)_{L_2(\Omega)} + (b,R_\f e)_{L_2(\Gamma_N)} - \B(u_{0,\ms} + \tilde b_\f  + g_h- R_\f g_h,R_\f e)\\
	&= (f,R_\f e)_{L_2(\Omega)},
	\end{align*}
	where have used the orthogonality \eqref{B-orthog} and the definitions \eqref{neumann-corr} and \eqref{Ritz-Vf} in the last equality. Now, since $R_\f e\in V_\f$ we have that $I_H R_\f e = 0$ and using \eqref{interpolation-global} we get
	\begin{align}\label{error-ideal}
	\alpha C_\mathrm{ko}^{-2} \|\nabla e&\|^2_{L_2(\Omega)} \leq \B(e,e) \leq (f,R_\f e-I_H R_\f e)_{L_2(\Omega)}\\& \leq \|f\|_{L_2(\Omega)}\|R_\f e -I_H R_\f e\|_{L_2(\Omega)} \leq C_\rho H\|f\|_{L_2(\Omega)}\|\nabla e\|_{L_2(\Omega)},\notag
	\end{align}
	and \eqref{gfemconv-est} follows.
\end{proof}

\section{Localization}\label{sec:loc}
The problem of finding $R_\f \lambda _z$ in \eqref{msbasis} is posed in the entire fine scale space $V_\f$ and thus computationally expensive. Moreover, the resulting basis functions may have global support. However, as we show in this section, the basis functions have exponential decay away from node $z$, which motivates a truncation of the basis functions. This truncation significantly reduces the computational cost and the resulting functions have local support.

We consider a localization strategy similar to the one proposed in \cite{Henning14}. We restrict the fine scale space $V_\f$ to patches $\omega_k(T)$ of coarse elements of the following type; for $T\in \mathcal T_H$ 
\begin{align*}
\omega_{0}(T)&:=\text{int } T,\\ \omega_{k}(T)&:=\text{int } \big(\cup\{\hat T \in \mathcal T_H: \hat T \cap \overline{\omega_{k-1}(T)}\neq \emptyset\}\big), \quad k=1,2,...
\end{align*}
Define $V_{\f }(\omega_k(T)):=\{v\in V_\f :v=0 \text{ on }(\overline\Omega\setminus\Gamma_N)\setminus \omega_k(K)\}$ to be the restriction of $V_\f $ to the patch $\omega_k(T)$. Note that the functions in $V_\f(\omega_k(T))$ are zero on the boundary $\partial \omega_k(T)\setminus \Gamma_N$.

We proceed by noting that the Ritz projection $R_\f$ in \eqref{Ritz-Vf} can be written as the sum
\begin{align*}
R_\f = \sum_{T\in \mathcal T_H} R_\f^T,
\end{align*}
where $R_\f^T:V_h \rightarrow V_\f$ and fulfills
\begin{align}\label{Ritz-Vf-T}
\B(R_\f^Tv,w)=\B(v,w)_T,\quad \forall w \in V_\f, \quad v \in V_h,\quad T \in \mathcal{T}_H,
\end{align}
where we define
\begin{align*}
\B(v,w)_T:=(A \varepsilon(v):\varepsilon(w))_{L_2(T)}, \quad T \in \mathcal{T}_H.
\end{align*}
We now aim to localize these computations by replacing $V_\f$ with $V_\f(\omega_k(T))$. Define $R_{\f,k}^T:V_h\rightarrow V_{\f}(\omega_k(T))$ such that 
\begin{align}\label{Ritz-Vfk}
\B(R_{\f,k}^Tv,w)=\B(v,w)_T,\quad \forall w \in V_\f(\omega_k(T)), \quad v \in V_h,\quad T \in \mathcal{T}_H,
\end{align}
and set $R_{\f,k}:=\sum_{T\in\mathcal T_H}R_{\f,k}^T$. We can now define the localized multiscale space
\begin{align}\label{mskbasis}
V_{\ms,k} = \{v_H-R_{\f,k} v_H : v_H \in V_H\}.
\end{align} 

Using the same techniques we also define localized versions of the Neumann boundary correctors \eqref{neumann-corr}. Note that $\tilde b_\f  = \sum _{T \in \mathcal T_H\cap \Gamma_N} \tilde b_\f ^T$ where $\tilde b_\f ^T$ is defined by
\begin{align*}
\B(\tilde b_\f ^T, w) = (b,w)_{L_2(\Gamma_N\cap T)}, \quad \forall w \in V_\f, \quad T\in \mathcal T_H,\quad  T \cap \Gamma_N\neq \emptyset,
\end{align*}
Thus, we define $\tilde b_{\f,k}^T \in V_\f(\omega_k(T))$ such that
\begin{align*}
\B(\tilde b_{\f,k}^T, w) = (b,w)_{L_2(\Gamma_N\cap T)}, \quad \forall w \in V_\f(\omega_k(T)), \quad T\in \mathcal T_H,\quad  T \cap \Gamma_N\neq \emptyset,
\end{align*}
and set $\tilde b_{\f,k} = \sum_{T\in\mathcal T_H} \tilde b_{\f,k}^T$.

We are now ready to define a localized version of \eqref{gfem}; find 
\begin{align*}
u_{\ms,k}=u_{0,\ms,k} + \tilde b_{\f,k} +  g_h - R_{\f,k} g_h,
\end{align*} 
such that $u_{0,\ms,k} \in V_{\ms,k}$ and
\begin{align}\label{gfem-loc}
\B(& u_{0,\ms,k},v)\\&=(f,v)_{L_2(\Omega)} + (b,v)_{L_2(\Gamma_N)}- \B(\tilde b_{\f,k} + g_h - R_{\f,k}g_h,v), \quad \forall v \in V_{\ms,k}.\notag
\end{align}
As for the non-localized problem \eqref{gfem}, we note that $\tilde b_{\f,k}$ and $R_{\f,k}$ vanish on $\Gamma_D$, so $\gamma u_{\ms,k}=\gamma g_h$, and
\begin{align*}
\B(u_{\ms,k},v) = (f,v)_{L_2(\Omega)} + (b,v)_{L_2(\Gamma_N)}, \quad \forall v \in V_{\ms,k}.
\end{align*}

The main result in this paper is the following theorem. 
\begin{theorem}\label{locgfemconv}
	Let $u_h$ be the solution to \eqref{fem} and $u_{\ms,k}$ the solution to \eqref{gfem}. Then there exists $\theta \in (0,1)$, depending on the contrast $\beta/\alpha$, such that
	\begin{align}\label{locgfemconv-est}
	\|u_h-u_{\ms,k}\|_{H^1(\Omega)} \leq  &CH\alpha^{-1}\|f\|_{L_2(\Omega)} \\&+ Ck^{d/2}\theta^k\sqrt{\frac{\beta^3}{\alpha^5}}(\|f\|_{L_2(\Omega)} + \|b\|_{L_2(\Gamma_N)}+\sqrt{\alpha}\|g_h\|_{\B(\Omega)}),\notag
	\end{align}
	where $C$ and $\theta$ depends on $C_\mathrm{ko}$, $\rho$, and $C_I$, but not on $k$, $h$, $H$, or the variations of $A$.
\end{theorem}

To prove the a priori bound in Theorem~\ref{locgfemconv} we first prove three lemmas. In the proofs we use the cut-off functions $\eta^T_{k} \in V_H$ with nodal values
\begin{subequations}\label{cutoff}
	\begin{align}
	\eta^T_{k}(x) &= 0, \quad \forall x \in \mathcal N \cap \overline{\omega_{k-1}(T)}, \\
	\eta^T_{k}(x) &= 1, \quad \forall x \in \mathcal N\cap (\Omega\setminus \omega_k(T)).
	\end{align}
\end{subequations}
These functions satisfy the following Lipschitz bound
\begin{align}\label{cutoff-bound}
\|\nabla \eta^T_{k}\|_{L_\infty(\Omega)} \leq CH^{-1}, \quad T \in \mathcal T_H,
\end{align}
where $C$ now depends on the quasi-uniformity. The proof technique relies on the multiplication of a function in the fine scale space $V_\f$ with a cut-off function. However, this product does not generally belong to the space $V_\f$. To fix this, let $\mathcal I_h:V\rightarrow V_h$ denote the classical linear Lagrange interpolation onto $V_h$. Using that $I_H$ in \eqref{interpolation} is a projection we get
\begin{align*}
z:=(I-I_H)\mathcal I_h (\eta^T_{k} w) \in V_\f(\Omega\setminus\omega_{k-2}(T)), \quad \forall w \in V_\f,
\end{align*}
where $I$ denotes the identity mapping. Note that the Lagrange interpolation is needed since $\eta^T_{k}w \not \in V_h$. Furthermore, we have $\supp \mathcal I_h(\eta^T_{k}w) \subseteq \Omega\setminus\omega_{k-1}(T)$ and $\supp I_H\mathcal I_h(\eta^T_{k}R^T_\f v) \subseteq \Omega\setminus\omega_{k-2}(T)$ and we conclude $z \in V_\f(\Omega\setminus\omega_{k-2}(T))$. 

\begin{lemma}\label{bound-annulus}
	For $w\in V_\f$ and $z:=(I-I_H)\mathcal I_h \eta^T_{k} w \in V_\f(\Omega\setminus\omega_{k-2}(T))$ it holds that $\supp(w-z) \subseteq \omega_{k}(T)$ and
	\begin{align}
	\|\nabla(w-z)\|_{L_2(\omega_k(T)\setminus\omega_{k-2}(T))} &\leq C_{I,\eta} \|\nabla w\|_{L_2(\omega_{k+1}(T)\setminus\omega_{k-3}(T))},\label{annulus-1}\\
	\|\nabla(w-z)\|_{L_2(\omega_k(T))} &\leq C'_{I,\eta} \|\nabla w\|_{L_2(\omega_{k+1}(T))},\label{annulus-2}\\
	\|\nabla z\|_{L_2(\Omega\setminus\omega_{k-2}(T))} &\leq C''_{I,\eta}\|\nabla w\|_{L_2(\Omega\setminus\omega_{k-3}(T))},\label{annulus-3}
	\end{align}	
	where $C_{I,\eta}$, $C'_{I,\eta}$, and $C''_{I,\eta}$ depends on $C_I$, $\rho$, and the bound in \eqref{cutoff-bound}, but not on $k$, $h$, $H$, $T$, or the variations of $A$.
\end{lemma}
\begin{proof}
	We have $\eta^T_{k}=1$ on $\Omega\setminus\omega_k(T)$ and hence
	\begin{align*}
	w-z = w - (I-I_H)w = 0, \quad \text{on } \Omega\setminus \omega_k(T),
	\end{align*}
	since $I_H w=0$ and it follows that $\supp(w-z) \subseteq \omega_k(T)$. 
	
	Now, note that
	\begin{align*}
	w-z = (I-I_H)(w-\mathcal I_h(\eta^T_kw)).
	\end{align*}
	Using the stability of $I_H$ in \eqref{interpolation-local} we derive the bound
	\begin{align*}
	\|\nabla(I-I_H)(w-\mathcal  I_h(\eta^T_kw))&\|_{L_2(\omega_k(T)\setminus\omega_{k-2}(T))}\\&\leq C_I\|\nabla(w-\mathcal I_h(\eta^T_kw))\|_{L_2(\omega_{k+1}(T)\setminus\omega_{k-3}(T))}.
	\end{align*}
	Now, using that the Lagrange interpolation $\mathcal I_h$ is $H^1$-stable for piecewise second order polynomials on shape regular meshes and the bound \eqref{cutoff-bound} we get
	\begin{align*}
	&\|\nabla \mathcal I_h (\eta^T_k w)\|_{L_2(\omega_{k+1}(T)\setminus\omega_{k-3}(T))} \leq C\|\nabla (\eta^T_{k}w)\|_{L_2(\omega_{k+1}(T)\setminus\omega_{k-3}(T))}\\
	&\quad\leq C\|w\nabla \eta^T_{k}\|_{L_2(\omega_{k}(T)\setminus\omega_{k-1}(T))} + C\|\eta^T_{k}\nabla w\|_{L_2(\omega_{k+1}(T)\setminus\omega_{k-1}(T))}\\
	&\quad\leq CH^{-1}\|w-I_Hw\|_{L_2(\omega_{k}(T)\setminus\omega_{k-1}(T))} + C\|\nabla w\|_{L_2(\omega_{k+1}(T)\setminus\omega_{k-1}(T))}\\
	&\quad\leq C\|\nabla w\|_{L_2(\omega_{k+1}(T)\setminus\omega_{k-2}(T))},
	\end{align*}
	where we also have utilized the bounded support of the cut-off function and the bound of $I_H$ in \eqref{interpolation-local}. This completes the bound \eqref{annulus-1}. The bounds in \eqref{annulus-2} and \eqref{annulus-3} follow similarly.
\end{proof}

\begin{lemma}\label{exp-decay-lemma}
	For the Ritz projection \eqref{Ritz-Vf} there exist $\theta \in (0,1)$, such that
	\begin{align}\label{exp-decay}
	\|\nabla R^T_\f v\|_{L_2(\Omega\setminus \omega_k(T))} \leq \theta^k \|\nabla R^T_\f v\|_{L_2(\Omega)}, \quad v \in V_h,
	\end{align}
	where $\theta$ depends on $\rho$ and the contrast $\beta/\alpha$, but not on $k$, $T$, $h$, $H$, or the variations of $A$.
\end{lemma}
\begin{proof}
	Fix an element $T\in\mathcal T_H$ and let $\eta^T_{k}$ be a cut-off function as in \eqref{cutoff}, and define $z$ as in Lemma~\ref{bound-annulus} with $w=R^T_\f v$ such that
	\begin{align}\label{support}
	z:=(I-I_H)\mathcal I_h (\eta^T_{k} R^T_\f v) \in V_\f(\Omega\setminus\omega_{k-2}(T)).
	\end{align}
	Since $\eta^T_{k}=1$ on $\Omega\setminus\omega_k(T)$, we have the identity $\mathcal I_h\eta^T_{k,l}R^T_\f v=R^T_\f v$ on $\Omega\setminus\omega_k(T)$. Using this and the bounds \eqref{B-bounds} for $\B(\cdot,\cdot)$ we get
	\begin{align}\label{expdecay:bound1}
	\|\nabla R^T_\f v\|^2_{L_2(\Omega\setminus \omega_k(T))}
	&=\|\nabla (I-I_H)R^T_\f v\|^2_{L_2(\Omega\setminus \omega_k(T))}\leq\|\nabla z\|^2_{L_2(\Omega)}\\&\leq C_\mathrm{ko}^2\alpha^{-1}\B(z,z).\notag
	\end{align}
	Now, due to \eqref{support} and \eqref{Ritz-Vf-T}, the following equality holds
	\begin{align*}
	\mathcal{B}(R^T_\f v,z) = \mathcal B(v, z)_T = 0,
	\end{align*}
	since $z$ does not have support on the element $T$. Using this and the fact that $\supp (z -R^T_\f v) \cap \supp z \subseteq \omega_k(T)\setminus \omega_{k-2}(T)$ we have
	\begin{align}\label{expdecay:bound2}
	\mathcal{B}(z, z)&= \mathcal{B}(z-R^T_\f v,z)= \int_{\omega_k(T)\setminus \omega_{k-2}(T)} A \varepsilon(z-R^T_\f v):\varepsilon(z)\,\mathrm{d}x\\
	&\leq 
	\beta \|\nabla (z-R^T_\f v)\|_{L_2(\omega_k(T)\setminus \omega_{k-2}(T))} \|\nabla z\|_{L_2(\omega_k(T)\setminus \omega_{k-2}(T))}\notag \\
	&\leq 
	\beta \|\nabla (z-R^T_\f v)\|_{L_2(\omega_k(T)\setminus \omega_{k-2}(T))} (\|\nabla (z-R^T_\f v)\|_{L_2(\omega_k(T)\setminus \omega_{k-2}(T))}\notag \\&\qquad+\|\nabla R^T_\f v\|_{L_2(\omega_k(T)\setminus \omega_{k-2}(T))})\notag\\
	&\stackrel{\eqref{annulus-1}}{\leq} C_{I,\eta}(C_{I,\eta}+1)\beta \|\nabla R^T_\f v\|^2_{L_2(\omega_{k+1}(T)\setminus\omega_{k-3}(T))},\notag
	\end{align}
	Combining \eqref{expdecay:bound1} and \eqref{expdecay:bound2} we have
	\begin{align*}
	\|\nabla R^T_\f v\|^2_{L_2(\Omega\setminus \omega_k(T))} &\leq C'\|\nabla R^T_\f v\|^2_{L_2(\omega_{k+1}(T)\setminus\omega_{k-3}(T))} \\&\leq C'(\|\nabla R^T_\f v\|^2_{L_2(\Omega\setminus\omega_{k-3}(T))}-\|\nabla R^T_\f v\|^2_{L_2(\Omega\setminus\omega_{k+1}(T))}),
	\end{align*}
	where $C'=C_\mathrm{ko}^2C_{I,\eta}(C_{I,\eta}+1)\beta/\alpha$. Thus
	\begin{align*}
	\|\nabla R^T_\f v\|^2_{L_2(\Omega\setminus \omega_{k+1}(T))} &\leq \frac{C'}{1+C'}\|\nabla R^T_\f v\|^2_{L_2(\Omega\setminus\omega_{k-3}(T))}.
	\end{align*}
	An iterative application of this result and relabeling $k+1 \rightarrow k$ yields \eqref{exp-decay}, with $\theta = ((\frac{C'}{1+C'})^{1/4})^{1/2} <1$.
\end{proof}

\begin{lemma}\label{Ritzconv-loc}
	For the Ritz projections \eqref{Ritz-Vf-T} and \eqref{Ritz-Vfk} we have the bound
\begin{align*}
	\|\sum_{T\in \mathcal T_H}\nabla(R^T_\f v - R^T_{\f,k}v)\|_{L_2(\Omega)}\leq Ck^{d/2}\theta^k\frac{\beta}{\alpha}\bigg(\sum_{T \in \mathcal T_H} \|\nabla R^T_\f v\|^2_{L_2(\Omega)}\bigg)^{1/2}, \quad v\in V_h,
\end{align*}
	with $\theta$ as in Lemma~\ref{exp-decay-lemma} and $C$ depends on $C_\mathrm{ko},C'_{I,\eta},$ and $C''_{I,\eta}$.
\end{lemma}

\begin{proof}
	Define $e_\f:=\sum_{T \in \mathcal T_H} R^T_\f v- R^T_{\f,k}v$ and let $\eta^T_{k+2}$ be the cut-off function as defined in \eqref{cutoff}. Since $e_\f \in V_\f$, we define $z^T_e:=(I-I_H)\mathcal I_h (\eta^T_{k+2} e_\f)$ as in Lemma~\ref{bound-annulus} and note that $\supp z^T_e \subseteq \Omega\setminus \omega_k(T)$. Thus, due to the fact that $\supp R^T_{\f,k}v \cap \supp z^T_e =\emptyset$ and \eqref{Ritz-Vf-T}, we have
	\begin{align*}
	\B(R^T_\f v- R^T_{\f,k}v,z^T_e) = \B(R^T_\f v, z^T_e) = \B(v, z^T_e)_T = 0.
	\end{align*}
	Using this and the bounds \eqref{B-bounds} we derive
	\begin{align}\label{bound-e-1}
	\|\nabla e_\f\|^2_{L_2(\Omega)} &\leq C_\mathrm{ko} \alpha^{-1} \B(e_\f,e_\f) = C_\mathrm{ko} \alpha^{-1}\sum_{T \in \mathcal{T}_H} \B(R^T_\f v-R^T_{\f,k}v,e_\f) \\&= C_\mathrm{ko} \alpha^{-1}\sum_{T \in \mathcal{T}_H} \B(R^T_\f v-R^T_{\f,k}v, e_\f - z^T_e).\notag\\
	&\leq C_\mathrm{ko} \sqrt{\beta}\alpha^{-1}\sum_{T \in \mathcal{T}_H} \|R^T_\f v-R^T_{\f,k}v\|_{\B(\Omega)} \|\nabla(e_\f - z^T_e)\|_{L_2(\omega_{k+2}(T))}. \notag
	\end{align}
	Now, we use Cauchy-Schwarz inequality for sums and Lemma~\ref{bound-annulus} to get
	\begin{align}\label{bound-e-2}
	&\sum_{T \in \mathcal{T}_H} \|R^T_\f v-R^T_{\f,k}v\|_{\B(\Omega)} \|\nabla(e_\f - z^T_e)\|_{L_2(\omega_{k+2}(T))} \\
	&\ \, \stackrel{\eqref{annulus-2}}{\leq} C'_{I,\eta} \Big(\sum_{T \in \mathcal{T}_H} \|R^T_\f v-R^T_{\f,k}v\|^2_{\B(\Omega)}\Big)^{1/2} \Big(\sum_{T\in\mathcal T_H}\|\nabla e_\f\|^2_{L_2(\omega_{k+3}(T))}\Big)^{1/2} \notag\\
	&\quad\leq C'_{I,\eta} C'_\rho k^{d/2} \Big(\sum_{T \in \mathcal{T}_H} \|R^T_\f v-R^T_{\f,k}v\|^2_{\B(\Omega)}\Big)^{1/2}\|\nabla e_\f\|_{L_2(\Omega)}.\notag
	\end{align}
	In the last inequality we have used the total number of patches overlapping an element $T$ is bounded by $C'_\rho k^{d/2}$, where $C'_\rho$ is a constant depending on the shape regularity of the mesh.
	
	It remains to bound $\|R^T_\f v-R^T_{\f,k}v\|_{\B(\Omega)}$. For this purpose we define $z_v=(I-I_H)\mathcal I_h (\eta^T_{k} R^T_\f v)$ as in Lemma~\ref{bound-annulus}. Recall that $R^T_\f v-z_v \in V_\f(\omega_k(T))$. Now, we use Galerkin orthogonality to derive
	\begin{align*}
	\|R^T_\f v-R^T_{\f,k}v\|_{\B(\Omega)} \leq \|R^T_\f v - w\|_{\B(\Omega)}, \quad \forall w \in V_\f(\omega_k(T)).
	\end{align*}
	Thus, with $w=R^T_\f v-z_v \in V_\f(\omega_k(T))$ we have
	\begin{align*}
	\|R^T_\f v-R^T_{\f,k}v\|_{\B(\Omega)} &\leq \|z_v\|_{\B(\Omega)} \leq \sqrt{\beta}\|\nabla z_v\|_{L_2(\Omega)} \leq \sqrt{\beta}\|\nabla z_v\|_{L_2(\Omega\setminus \omega_{k-2})} \\&\leq C''_{I,\eta}\sqrt{\beta}\|\nabla R^T_\f v\|_{L_2(\Omega\setminus\omega_{k-3})}.
	\end{align*}
	Using Lemma~\ref{exp-decay-lemma} we thus have
	\begin{align}\label{bound-RT}
	\|\nabla(R^T_\f v-R^T_{\f,k}v)\|_{L_2(\Omega)} \leq  C''_{I,\eta}\sqrt{\beta}\theta^k\|\nabla R^T_\f v\|_{L_2(\Omega)}.
	\end{align}
	Combining \eqref{bound-e-1}, \eqref{bound-e-2}, and \eqref{bound-RT}, concludes the proof.
\end{proof}

\begin{remark}
	Using the same techniques as in Lemma~\ref{exp-decay-lemma} and Lemma~\ref{Ritzconv-loc} we can prove (since the right hand side still has support only on a triangle $T\in \mathcal T_H$) exponential decay also for the Neumann boundary correctors
	\begin{align*}
	\|\nabla(\tilde b^T_{\f} - \tilde b^T_{\f,k})\|_{L_2(\Omega)} \leq Ck^{d/2}\theta^k\frac{\beta}{\alpha}\bigg(\sum_{T \in \mathcal T_H} \|\nabla \tilde b^T_\f \|^2_{L_2(\Omega)}\bigg)^{1/2}, \quad v\in V_h,
	\end{align*}
	with $\theta$ as in Lemma~\ref{exp-decay-lemma}.
\end{remark}

We are now ready to prove Theorem~\ref{locgfemconv}.
\begin{proof}[Proof of Theorem~\ref{locgfemconv}]
	Recall that $u_h=u_{0,h} + g_h$ and $u_{\ms,k} = u_{0,\ms,k} + \tilde b_{\f,k} + g_h - R_{\f,k}g_h$. Due to \eqref{fem} and \eqref{gfem-loc} we have the Galerkin orthogonality
	\begin{align*}
	\B(u_h-u_{\ms,k},v) = 0, \quad \forall v \in V_{\ms,k},
	\end{align*} 
	which implies
	\begin{align*}
	\|u_h-u_{\ms,k}\|_{\B(\Omega)}\leq \|u_h-v-\tilde b_{\f,k} - g_h +R_{\f,k}\|_{\B(\Omega)}, \quad \forall v \in V_{\ms,k}.
	\end{align*}
	Let $u_\ms=u_{0,\ms} + \tilde b_\f + g_h - R_\f g_h$ be the solution to \eqref{gfem}. Since $u_{0,\ms}\in V_\ms$ and $u_{0,\ms,k}\in V_{\ms,k}$ there exist $v_H, v_{H,k} \in V_H$, such that
	\begin{align*}
	u_{0,\ms} = v_H - R_\f v_H, \quad u_{0,\ms,k} = v_{H,k} - R_{\f,k} v_{H,k}.
	\end{align*}
	Using the Galerkin orthogonality with $v = v_H - R_{\f,k} v_H \in V_{\ms,k}$ we have
	\begin{align*}
	&\|u_h - u_{\ms,k}\|_{\B(\Omega)} \leq \|u_h - v_H + R_{\f,k}v_H-\tilde b_{\f,k} - g_h +R_{\f,k}g_h\|_{\B(\Omega)} \\
	& \quad \leq \|u_h - v_H + R_\f v_H -\tilde b_\f - g_h +R_\f g_H\|_{\B(\Omega)} + \|R_{\f,k}v_H - R_\f v_H\|_{\B(\Omega)} \\&\qquad+ \|\tilde b_{\f,k} - \tilde b_\f\|_{\B(\Omega)} + \|R_{\f,k}g_h - R_\f g_h\|_{\B(\Omega)}),
	\end{align*}
	From \eqref{error-ideal} in Lemma~\ref{gfemconv} we have
	\begin{align*}
	\|u_h - v_H + R_\f v_H -\tilde b_\f -g_h +R_\f g_H\|_{\B(\Omega)} &= \|u_h - u_\ms\|_{\B(\Omega)} \\&\leq C_\rho C_\mathrm{ko}/\sqrt{\alpha}H\|f\|_{L_2(\Omega)},
	\end{align*}
	and due to Lemma~\ref{Ritzconv-loc} and \eqref{Ritz-Vf-T} we have
	\begin{align*}
	\|R_{\f,k}v_H-R_\f v_H\|^2_{\B(\Omega)}&\leq \beta\|\nabla(R_{\f,k}v_H-R_\f v_H)\|^2_{L_2(\Omega)}\\&\leq C\beta^3/\alpha^2 k^d\theta^{2k}\sum_{T\in\mathcal T_H}\|\nabla R^T_\f v_H\|^2_{L_2(\Omega)}\\
	&\leq C\beta^3/\alpha^2 k^d\theta^{2k}\sum_{T\in\mathcal T_H}\|\nabla v_H\|^2_{L_2(T)} 
	\\&= C\beta^3/\alpha^2 k^d\theta^{2k}\|\nabla v_H\|^2_{L_2(\Omega)}.
	\end{align*}
	Now, since $u_{0,\ms}$ satisfies \eqref{gfem} we deduce the stability estimate
	\begin{align*}
	\|u_{0,\ms}\|_{\B(\Omega)} &\leq C(1/\sqrt{\alpha}(\|f\|_{L_2(\Omega)} + \|b\|_{L_2(\Gamma_N)}) + \|\tilde b_\f\|_{\B(\Omega)} + \|g_h-R_\f g_h\|_{\B(\Omega)})\\
	&\leq C/\sqrt{\alpha}(\|f\|_{L_2(\Omega)} + \|b\|_{L_2(\Gamma_N)} + \sqrt{\alpha}\|g_h\|_{\B(\Omega)}),
	\end{align*}
	where we have used stability derived from \eqref{neumann-corr} and \eqref{Ritz-Vf} in the last inequality. Hence, using that $I_HR_\f v_H=0$ and the stability of $I_H$ \eqref{interpolation-global}, we get
	\begin{align*}
	\|\nabla v_H\|_{L_2(\Omega)}&=\|\nabla I_H(v_H - R_\f v_H)\|_{L_2(\Omega)}\leq C\|\nabla u_{0,\ms}\|_{L_2(\Omega)} \leq C/\sqrt{\alpha}\|u_{0,\ms}\|_{\B(\Omega)} \\
	&\leq C/\alpha(\|f\|_{L_2(\Omega)} + \|b\|_{L_2(\Gamma_N)} + \sqrt{\alpha}\|g_h\|_{\B(\Omega)}).
	\end{align*}
	Similarly, we deduce the bounds
	\begin{align*}
	\|\tilde b_{\f,k} - \tilde b_\f\|^2_{\B(\Omega)} &\leq C\beta^3/\alpha^2k^d\theta^{2k} \sum_{\substack{T\in\mathcal T_H\\T\cap \Gamma_N \neq \emptyset}}\|\nabla \tilde b^T_\f\|^2_{L_2(\Gamma_N)}\\&\leq C\beta^3/\alpha^4k^d\theta^{2k}\|b\|^2_{L_2(\Gamma_N)}.\\
	\|R_{\f,k}g_h - R_\f g_h\|^2_{\B(\Omega)} &\leq C\beta^3/\alpha^2k^d\theta^{2k} \sum_{T\in\mathcal T_H}\|\nabla R^T_\f g_h\|^2_{L_2(\Omega)}\\&\leq C\beta^3/\alpha^3 k^d\theta^{2k}\|g_h\|^2_{\B(\Omega)}.
	\end{align*}
	Thus we have
	\begin{align*}
	\|\nabla &(u_h-u_{\ms,k})\|_{L_2(\Omega)}\leq C_\mathrm{ko}/\sqrt{\alpha}\|u_h-u_{\ms,k}\|_{\B(\Omega)} \\&\leq C/\alpha H\|f\|_{L_2(\Omega)}+ C\sqrt{\beta^3/\alpha^5}k^{d/2}\theta^{k}(\|f\|_{L_2(\Omega)} + \|b\|_{L_2(\Gamma_N)} + \sqrt{\alpha}\|g_h\|_{\B(\Omega)}).
	\end{align*}
	The proof is now complete.
\end{proof}

\begin{remark}
	To achieve linear convergence in Theorem~\ref{locgfemconv} the size of the patches for the localization should be chosen proportional to $\log H^{-1}$, i.e. $k=c\log(H^{-1})$ for some constant $c$.  	
\end{remark}

\section{Numerical Experiment}\label{sec:experiments}
In this section we perform two numerical experiments to test the convergence rate obtained in Theorem~\ref{locgfemconv}. The first experiment shows that linear convergence is obtained, in the $H^1$-norm, for a problem with multiscale data. The second experiment shows that the locking effect is reduced for a problem with high value of $\lambda$. We refer to \cite{Henning16} for a discussion on how to implement this type of generalized finite elements efficiently.

We consider an isotropic medium, see Remark~\ref{isotropic}, on the unit square in $\R^2$. Recall that the stress tensor in the isotropic case takes the form
\begin{align*}
\sigma(u)=2\mu\varepsilon(u)+\lambda(\nabla\cdot u)I,
\end{align*}
where $\mu$ and $\lambda$ are the Lam\'{e} coefficients. For simplicity we consider only homogeneous Dirichlet boundary conditions, that is, $\Gamma_D=\partial \Omega$ and $g=0$. The body forces are set to $f=[1\ 1]^\intercal$. 

In the first experiment, we test the convergence on two different setups for the Lam\'{e} coefficients, one with multiscale features, and one with constant coefficients $\mu=\lambda=1$. For the problem with multiscale features we choose $\mu$ and $\lambda$ to be discontinuous on a Cartesian grid of size $2^{-5}$. The values at the cells are chosen randomly between $0.1$ and $10$. The resulting coefficients are shown in Figure~\ref{lame-multiscale}.

\begin{figure}[h]
	\centering
	\begin{subfigure}[b]{0.48\textwidth}
		\includegraphics[width=\textwidth]{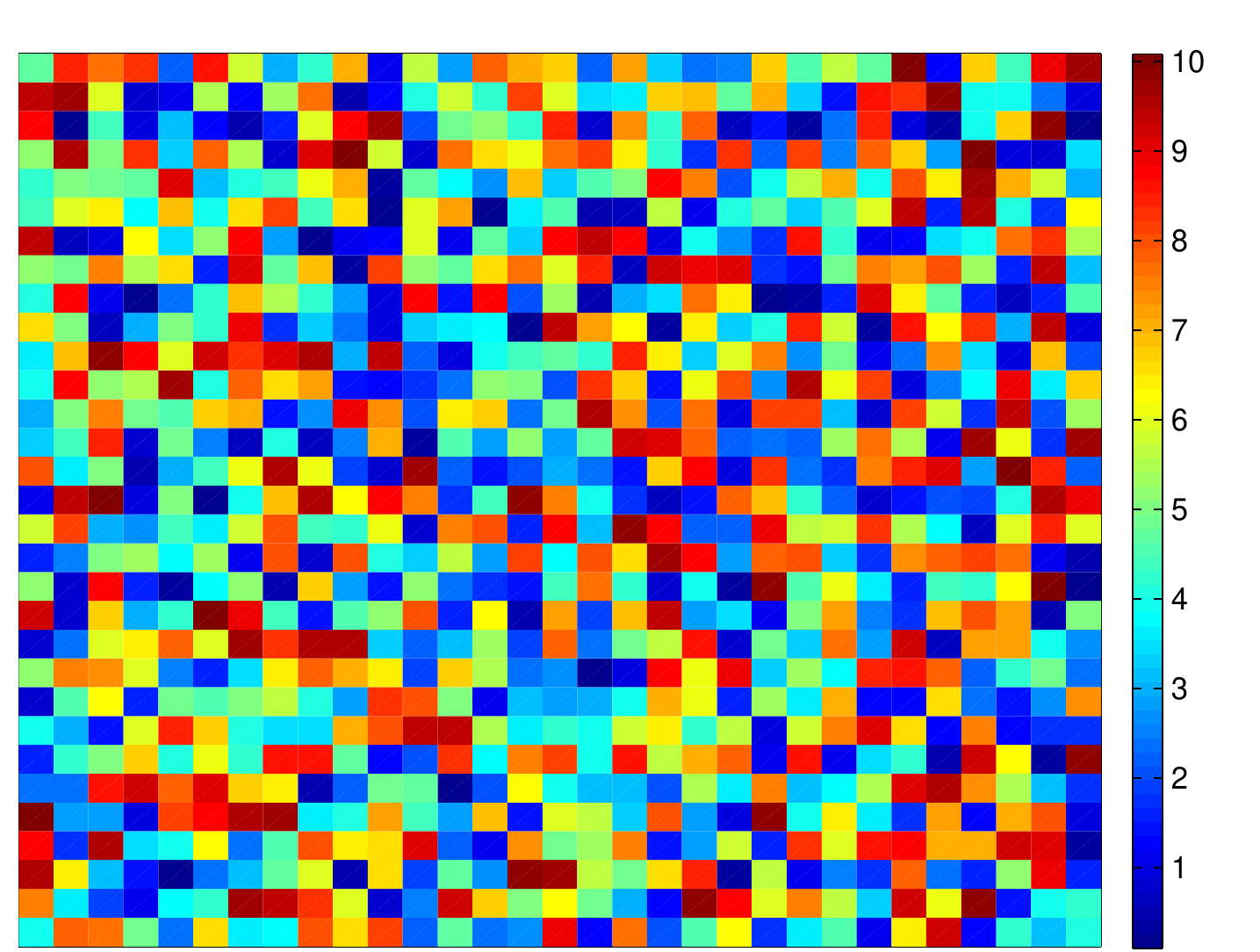}
		\caption{Lam\'{e} coefficient $\mu$}
	\end{subfigure}
	~
	\begin{subfigure}[b]{0.48\textwidth}
		\centering
		\includegraphics[width=\textwidth]{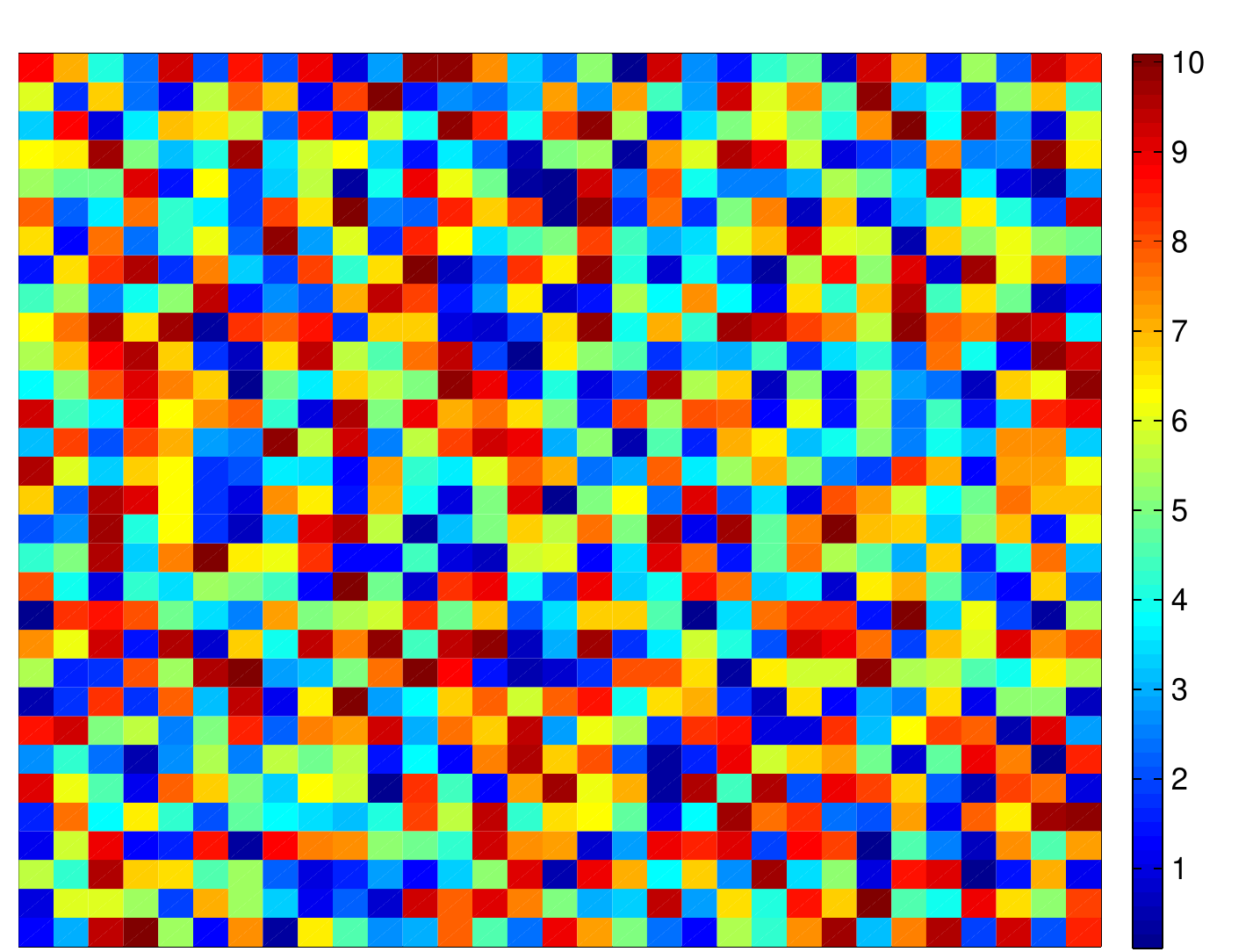}
		\caption{Lam\'{e} coefficient $\lambda$} 
	\end{subfigure}
	\caption{Lam\'{e} coefficients with multiscale features.}\label{lame-multiscale}
\end{figure}

For the numerical approximations we discretize the domain with a uniform triangulation. The reference solution $u_h$ in \eqref{fem} is computed using a mesh of size $h=\sqrt{2}\cdot2^{-6}$, which is small enough to resolve the multiscale coefficients in Figure~\ref{lame-multiscale}. The generalized finite element (GFEM) solution in \eqref{gfem-loc} is computed on several meshes of decreasing size, $H=\sqrt{2}\cdot 2^{-1},...,\sqrt{2}\cdot 2^{-5}$ with $k=1,1,2,2,3$, which corresponds to $k=\lceil 0.8\log H^{-1}\rceil$. These solutions are compared to the reference solution. For comparison we also compute the classical piecewise linear finite element (P1-FEM) solution on the meshes of size $H=\sqrt{2}\cdot 2^{-1},...,\sqrt{2}\cdot 2^{-5}$. The error is computed using the $H^1$ semi-norm $\|\nabla \cdot\|$ and plotted in Figure~\ref{plot-error}.

\begin{figure}[h]
	\centering
	\begin{subfigure}[b]{0.48\textwidth}
		\includegraphics[width=\textwidth]{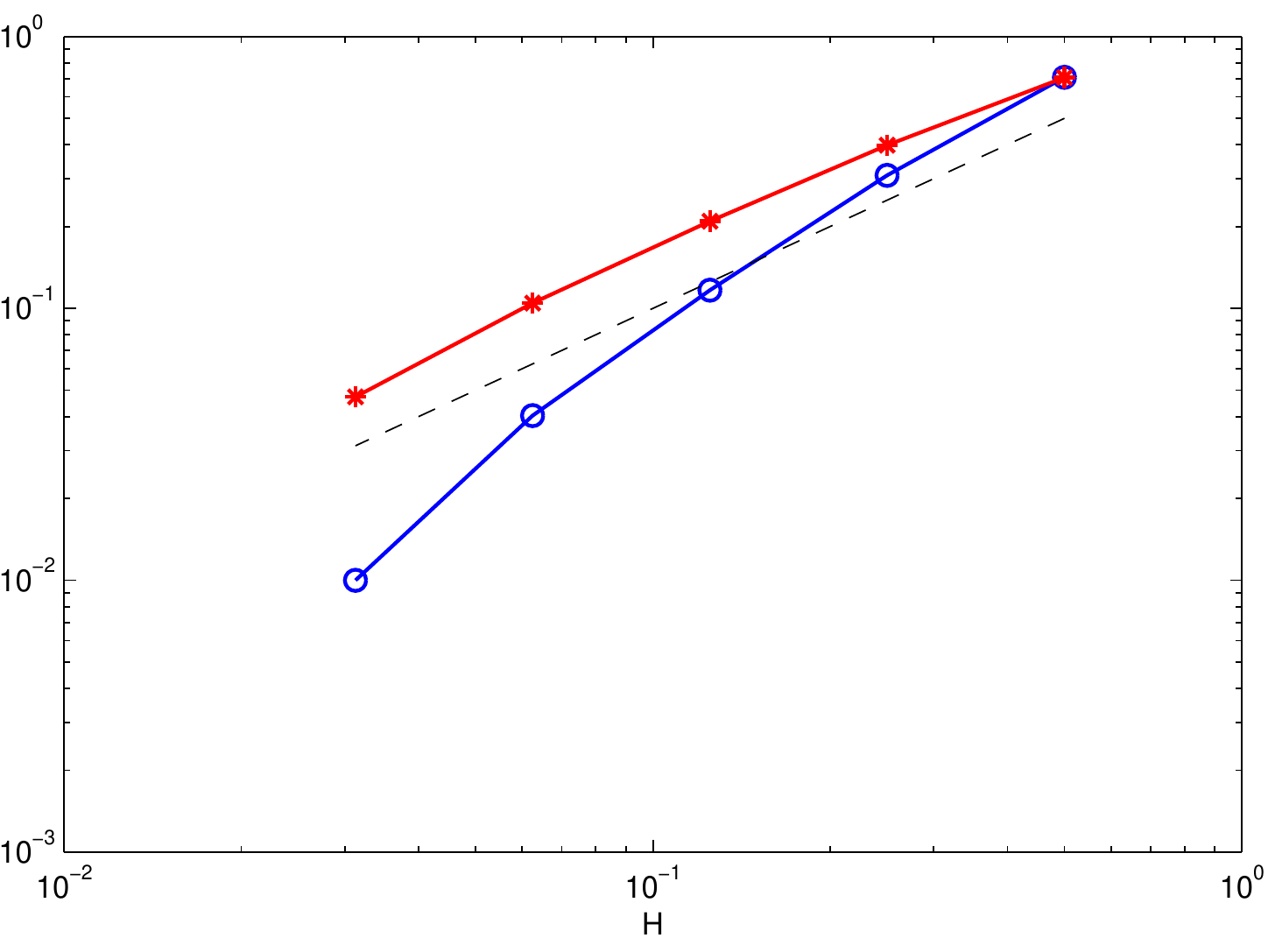}
		\caption{Constant coefficients $\mu=\lambda=1$.}
	\end{subfigure}
	~
	\begin{subfigure}[b]{0.48\textwidth}
		\centering
		\includegraphics[width=\textwidth]{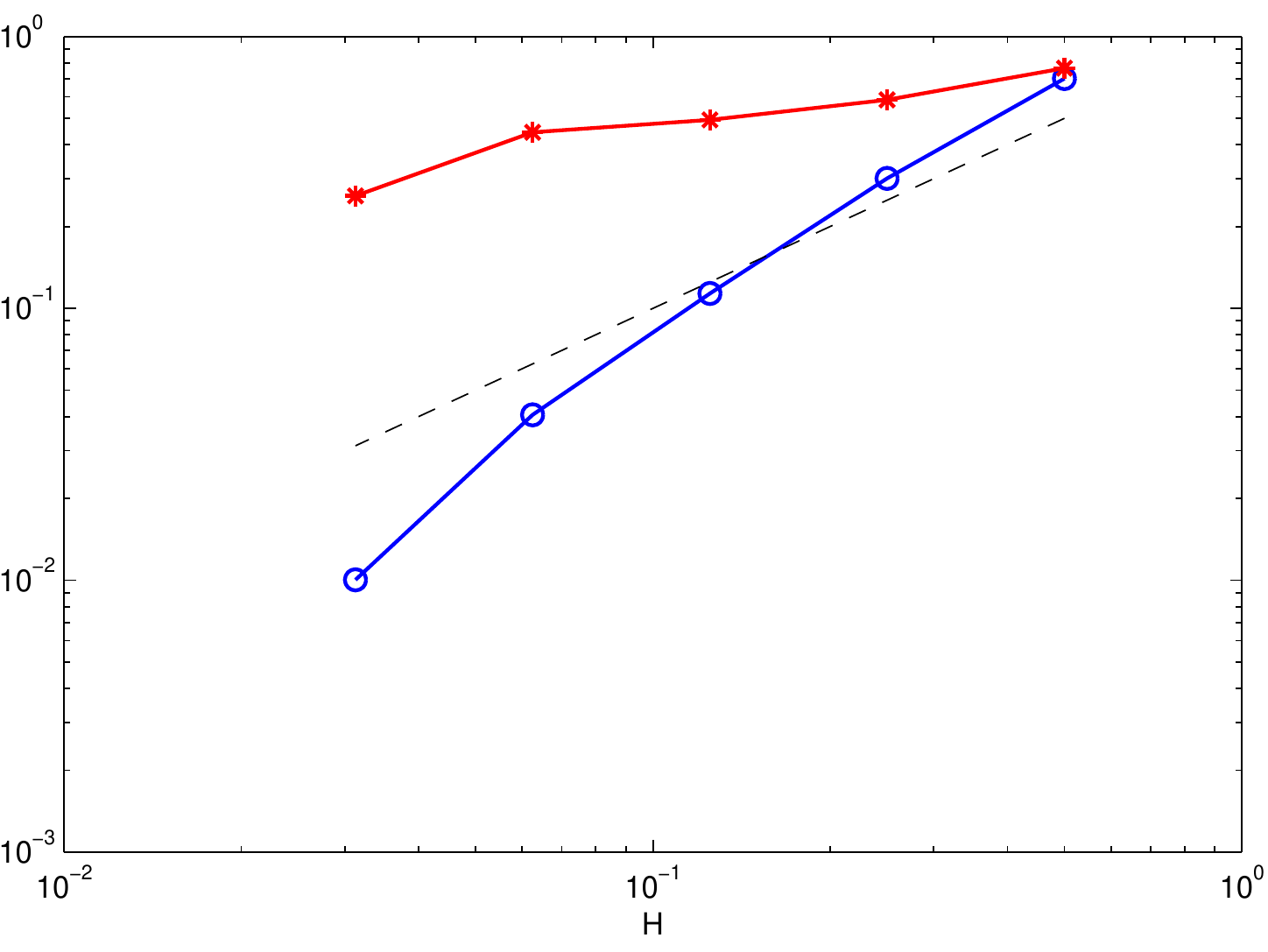}
		\caption{Multiscale coefficients, see Figure~\ref{lame-multiscale}.} 
	\end{subfigure}
	\caption{Relative errors using GFEM (blue $\circ$) and P1-FEM (red $\ast$) for the linear elasticity problem plotted against the mesh size $H$. The dashed line is $H$.}\label{plot-error}
\end{figure}

In Figure~\ref{plot-error} we see that both methods, as expected, show linear convergence for the problem with constant coefficients. For the problem with multiscale coefficients we clearly see the advantages with the generalized finite element method, which shows linear convergence also in this case, while the classical finite element shows far from optimal convergence. 

For the second experiment we aim to test the locking effect. We consider a problem from \cite{Brenner93}. The domain is set to the unit square $\Omega=[0,1]\times [0,1]$ and $g_D=0$ on the boundary $\Gamma_D=\partial \Omega$. Furthermore, with $\mu=1$ and the right hand side $f=[f_1 \ f_2]^\intercal$ chosen as
\begin{align*}
f_1 &= \pi^2\Big( 4\sin (2\pi y)(-1+2\cos(2\pi x)) -\cos \pi(x+y) +\frac{2}{1+\lambda}\sin(\pi x)\sin(\pi y)\Big),\\
f_2 &= \pi^2\Big( 4\sin (2\pi y)(1-2\cos(2\pi x)) -\cos \pi(x+y) +\frac{2}{1+\lambda}\sin(\pi x)\sin(\pi y)\Big),
\end{align*}
the exact solution $u=[u_1 \ u_2]^\intercal$ is given by
\begin{align*}
u_1 &= \sin (2\pi y)(-1+2\cos(2\pi x)) + \frac{1}{1+\lambda}\sin(\pi x)\sin(\pi y),\\
u_2 &= \sin (2\pi y)(1-2\cos(2\pi x)) + \frac{1}{1+\lambda}\sin(\pi x)\sin(\pi y).
\end{align*}
In this experiment we let $\lambda=10^3$. The discretization of the domain remain the same as in our first example, but the size of the reference mesh is set to $h=\sqrt{2}\cdot 2^{-7}$ which is sufficiently small for $u_h$ to be a relatively good approximation, since $h<1/\sqrt{\lambda}$. Indeed, using the knowledge of the exact solution we have $\|\nabla (\mathcal I_h(u)-u_h)\|_{L_2(\Omega)}/\|\nabla \mathcal I_h(u)\|_{L_2(\Omega)}\approx 0.15$, where $\mathcal I_h$ is the Lagrangian nodal interpolation onto $V_h$. 

The GFEM and the classical P1-FEM solutions are computed for the values $H=\sqrt{2}\cdot 2^{-1},...,\sqrt{2}\cdot 2^{-6}$. The localization parameter is chosen to be $k=1,1,2,2,3,4$ which corresponds to $k=\lceil 0.8\log H^{-1}\rceil$. The numerical approximations $u_{\ms,k}$ and $u_H$ are compared to the reference solution $u_h$ and the error is computed using the $H^1$-seminorm. The relative errors are plotted in Figure~\ref{plot-locking}. Clearly, the classical finite element method suffers from locking effects for the coarser mesh sizes. However, the generalized finite element solution shows linear convergence, that is, no locking effect is noted. 

\begin{figure}[h]
		\centering
		\includegraphics[width=0.5\textwidth]{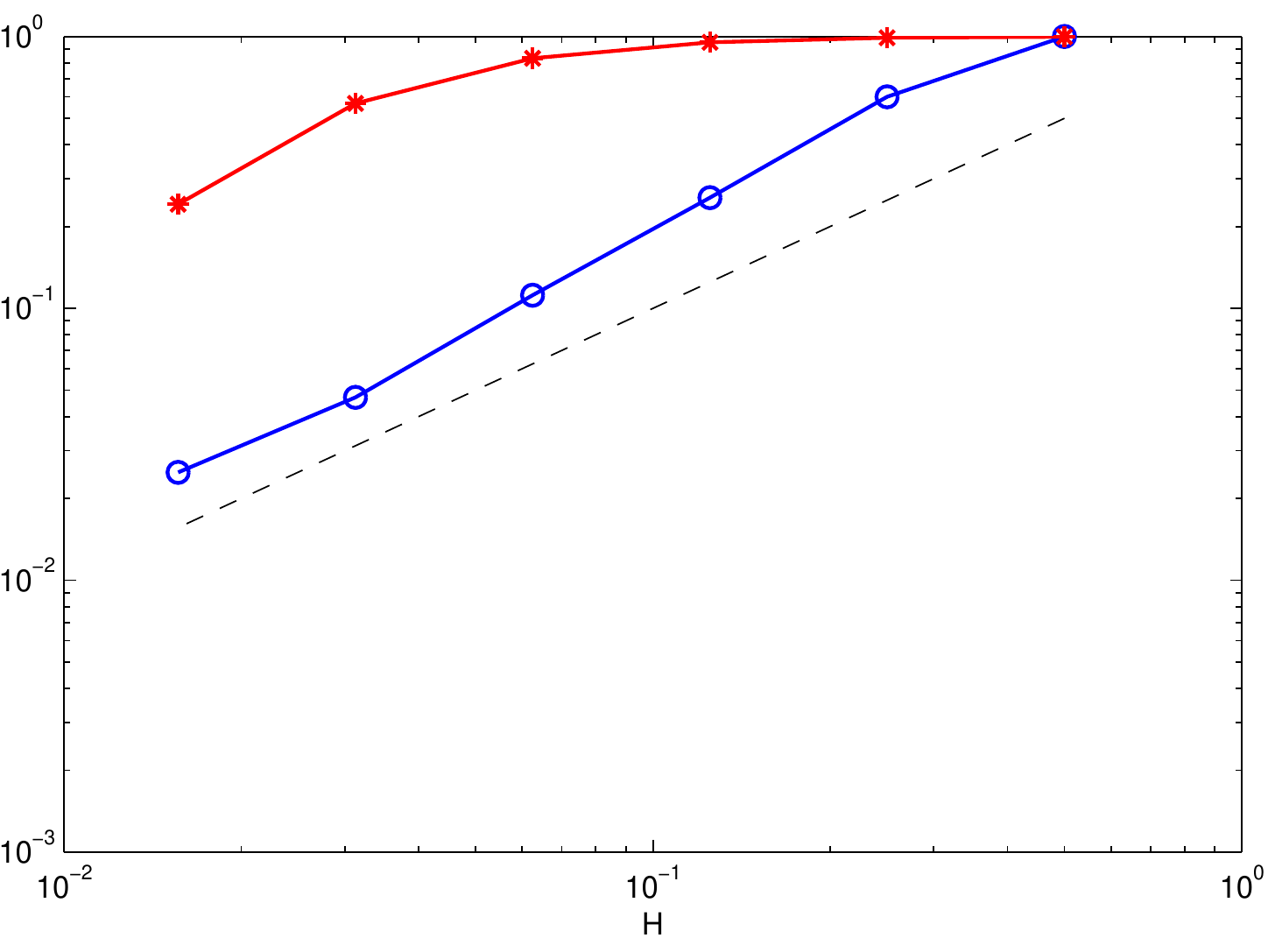}
	\caption{Relative errors for the locking problem using GFEM (blue $\circ$) and P1-FEM (red $\ast$)  plotted against the mesh size $H$. The dashed line is $H$.}\label{plot-locking}
\end{figure}

\bibliographystyle{abbrv}
\bibliography{elasticity_ref}
\end{document}